\documentclass[12pt]{amsart}
\usepackage{amsmath,amsfonts,amssymb}

\theoremstyle{plain}

\newtheorem{thm}{Theorem}
\newtheorem{lem}{Lemma}

\title[Integers as combinations of powers]
{Representing integers as linear combinations of power products}

\author{Lajos Hajdu}
\address{Institute of Mathematics\\
University of Debrecen\\
H-4010 Debrecen, P.O. Box 12\\
Hungary}
\email{hajdul@math.unideb.hu}

\author{Rob Tijdeman}
\address{Mathematical Institute\\
Leiden University\\
2300 RA Leiden, P.O. Box 9512\\
The Netherlands}
\email{tijdeman@math.leidenuniv.nl}



\thanks{Research supported in part by the OTKA grant K75566, and by the T\'AMOP 4.2.1./B-09/1/KONV-2010-0007 project. The project is implemented through the New Hungary Development Plan, cofinanced by the European Social Fund and the
European Regional Development Fund.}

\begin{document}

\begin{abstract}
Let $P$ be a finite set of at least two prime numbers, and $A$ the set of positive integers that are products of powers of primes from $P$. Let $F(k)$ denote the smallest positive integer which cannot be presented as sum of less than $k$ terms of $A$. In a recent paper Nathanson asked to determine the properties of the function $F(k)$, in particular to estimate its growth rate. In this paper we derive several results on $F(k)$ and on the related function $F_{\pm}(k)$ which denotes the smallest positive integer which cannot be presented as sum of less than $k$ terms of $A \cup (-A)$.
\end{abstract}

\maketitle

\section{Introduction}

Let $P$ be a nonempty finite set of at least two prime numbers, and $A$ the set of positive integers that are products of powers of primes from $P$. Put $A_{\pm} =A \cup (-A)$. Then there does not exist an integer $k$ such that every positive integer can be represented as a sum of at most $k$ elements of $A_{\pm}$. This follows e.g. from Theorem 1 of Jarden and Narkiewicz \cite{jn}, cf. \cite{h, ahl}. At a conference in Debrecen in 2010 Nathanson announced the following stronger result (see also \cite{n}):
\\
\\
{\it For every positive integer $k$ there exist infinitely many integers $n$ such that $k$ is the smallest value of $l$ for which $n$ can be written as
$$
n= a_1 + a_2 + \cdots + a_l ~~(a_1, a_2, \dots, a_l \in A_{\pm}).
$$
}
\\
Let $F(k)$ be the smallest positive integer which cannot be
presented as a sum of less than $k$ terms of $A$ and
 $F_{\pm}(k)$ the smallest positive integer which cannot be
presented as a sum of less than $k$ terms of $A_{\pm}$. Problem 2 of \cite{n} Nathanson is to give estimates for $F(k)$.
(The notation in \cite{n} is different from ours.)
Problem 1 is the corresponding question for $F_{\pm}(k)$ in case $A$ consists of the pure powers of 2 and of 3.

In \cite{ht} two of the authors considered Problem 1 in the more general setting of powers of any finite set of positive integers. They gave lower and upper bounds for $F(k)$ and $F_{\pm}(k)$. In the present paper we consider Problem 2. We give lower and upper bounds for $F(k)$ and $F_{\pm}(k)$ for $A$ as defined above.

We show that there exists an effectively computable number $c$ depending only on $P$, an effectively computable number $C$ depending only on $\varepsilon$ and an effectively computable constant $C_{\pm}$ such that
$k^{ck}<F(k) < C(kt)^{(1 + \varepsilon)kt}$ and $k^{ck}<F_{\pm}(k) < \exp((kt)^{C_{\pm}})$.
The method of proof is an adaptation of that in \cite{ht}, but in the case of the lower bound an additional argument is needed. For the upper bound we need an extended version of a theorem of \'Ad\'am, Hajdu and Luca \cite{ahl} in which a result of Erd\H{o}s, Pomerance and Schmutz \cite{eps} plays an important part. We state the  result of Erd\H{o}s, Pomerance and Schmutz and its refinement in Section \ref{sec2} and our generalization of the result of \'Ad\'am, Hajdu and Luca in Section \ref{sec3}. In Section \ref{sec4} we derive the lower and upper bounds for $F(k)$ and $F_{\pm}(k)$. In Section 5 we apply the Qualitative Subspace Theorem to prove that for some number $c^*$ depending only on $P, k$ and $\varepsilon$ the inequality $F_{\pm} (k) \leq (kt)^{(1+ \varepsilon)kt}$ holds for $k > c^*$.

\vskip1cm

\section{An extension of a theorem of Erd\H{o}s, Pomerance and
Schmutz}
\label{sec2}

Let $\lambda(m)$ be the Carmichael function of the positive
integer $m$, that is the least positive integer for which
$$
b^{\lambda(m)} \equiv 1 ~~({\rm mod}~m)
$$
for all $b \in \mathbb{Z}$ with gcd$(b,m)=1$. Theorem 1 of
\cite{eps} gives the following information on small values of the Carmichael function.

\vskip.1cm

\noindent {\it For any increasing sequence $(n_i)_{i=1}^{\infty}$ of positive integers, and any positive constant $C_1 < 1/ \log 2$, one has
$$
\lambda (n_i) > ( \log n_i)^{C_1 \log \log \log n_i}
$$
for $i$ sufficiently large. On the other hand, there exist a
strictly increasing sequence $(n_i)_{i=1}^{\infty}$ of positive integers and a positive constant $C_2$, such that, for every $i$,
$$
\lambda (n_i) < ( \log n_i)^{C_2 \log \log \log n_i}.
$$}

This nice result does not give any information on the size of
$n_i$. For our purposes the following quantitative version will be needed.

\vskip.1cm

\begin{lem} [\cite{ht}, Theorem 1]
\label{thm1} There exist positive constants $C_3, C_4$ such that
for every large integer $i$ there is an integer $m$ with
$$ \log m
\in [\log i, (\log i)^{C_3}] {\rm ~ and ~}
\lambda(m) < ( \log m)^{C_4 \log \log \log m}.
$$
\end{lem}

\section{An extension of a theorem of \'Ad\'am, Hajdu and Luca}

\label{sec3}

Let $k$ be a positive integer. Put
$$
H_{P,k} = \{n \in \mathbb{Z}: n= \sum_{i=1}^l a_i ~~{\rm with}~~ l \leq k\}
$$
where $a_i \in A ~~  (i=1, 2, \dots,k)$. For $H
\subseteq \mathbb{Z}$ and $m \in \mathbb{Z}, m \geq 2$, we write
$\sharp H$ for the cardinality of the set $H$ and
$$
H ({\rm mod}~m) = \{i: 0 \leq i < m, h \equiv i ~( {\rm mod}~ m)~
{\rm for ~ some~} h \in H \}.
$$

The next theorem is a generalization of a result from \cite{ahl}.

\vskip .3cm

\begin{thm}
\label{thm1} Let $C_3, P$ and $k$ be given as above. There is a constant $C_5$ such that for every sufficiently large integer $i$ there exists an integer $m$ with $\log m \in [\log i, (\log i)^{C_3}]$ and
$$
\sharp H_{P, k} ~({\rm mod}~m) < (\log m)^{C_5kt \log\log \log m}.
$$
\end{thm}

\vskip.2cm

In the proof of Theorem \ref{thm1} the following lemma is used.

\vskip.1cm

\begin{lem}
{\rm (\cite{ahl}, Lemma 1)}.  \label{lem1}
\noindent Let $m=q_1^{\alpha_1} \cdots q_z^{\alpha_z}$ where
$q_1, \dots, q_z$ are distinct primes and $\alpha_1, \dots,
\alpha_z$ positive integers, and let $b \in \mathbb{Z}$. Then
$$
\sharp \{b^u~({\rm mod}~m) : u \geq 0 \} \leq \lambda (m) +
\max_{1 \leq j \leq z} \alpha_j.
$$
\end{lem}

\vskip.1cm

\begin{proof}[Proof of Theorem $1$.] Let $i$ be a large
integer. Choose $m$  according to Lemma \ref{thm1}. Write $m$ as a product of powers of distinct primes as in Lemma \ref{lem1}. Lemma \ref{lem1} implies that
$$
\sharp \{h~({\rm mod}~m) : h \in H_{P,k} \}
\leq  \left(\lambda (m) + \max_{1 \leq j \leq z} \alpha_j + 1\right)^{kt}.
$$
On the other hand, with the constant $C_4$ from Lemma \ref{thm1},
$$
\lambda (m) + \max_{1 \leq j \leq z} \alpha_j < (\log m)^{C_4 \log \log \log m} + \frac{ \log m}{ \log 2} .
$$
The combination of both inequalities yields the theorem.
\end{proof}

\section{Effective results on combinations of power products}

\label{sec4}

Suppose we want to express the positive integer $n$ as a finite sum of elements of $A$. For this we  apply the greedy algorithm. If we subtract the largest element of $A$ not exceeding $n$ from $n$, we are left with a rest which is less than $n/(\log n)^{c_1}$ for some number $c_1>0$ depending only on the two smallest elements of $P$ according to \cite{t}.
We can iterate subtracting the largest element of $A$ not exceeding the rest from the rest and as long as the rest exceeds $\exp (\sqrt { \log n})$ reduce the rest each time by a factor at least $(\log n)^{c_1/2}$. If the rest is smaller than  $\exp (\sqrt { \log n})$ we can reduce the rest each step by a factor larger than some constant $c_2>1$, with $c_2$ depends only on the smallest prime from $P$. Thus we find that the sum of $$k \leq  \frac{2\log n} { c_1  \log \log n} + \frac {\sqrt{ \log n}} {\log c_2} $$ elements of $A$ suffices to represent $n$. This implies the lower bound $k^{ck}$ for $F(k)$ in Theorem 2(i) below. Of course, $F(k) \leq F_{\pm}(k)$ for all $k$.

For an upper bound for $F(k)$ we study the number of representations of positive integers up to $n$ as $\sum_{j=1}^{l} a_j$ with $a_j \in A, l \leq k$.
Since the number of elements of $A \cup \{0\}$ not exceeding $n$ is at most $(C_6 \log n)^t$, the number of represented integers is at most $(C_6 \log n )^{kt}$. If this number is less than $n$, then we are sure that some positive integer $\leq n$ is not represented. This is the case if
$$
kt < \frac{ \log n }{  \log \log n + \log C_6}.
$$
Suppose $n> (kt)^{(1 + \varepsilon)kt}$.
Then it follows from the monotonicity of the function $\log x / (\log \log x + C_6)$ for large $x$ that
$$
\frac { \log n} {\log \log n + C_6} > \frac {(1 + \varepsilon) kt \log kt} { \log (kt) + \log((1+ \varepsilon) \log (kt)) + C_6} > kt
$$
for $kt$ sufficiently large. By choosing $C_7$ suitably for the smaller values of $kt$, it suffices for all values of $kt$ that $n \geq C_7(kt)^{ (1 + \varepsilon)kt} $. Thus
$$
F(k) \leq C_7(kt)^{(1 + \varepsilon)kt}.
$$

Next we consider representations by sums of elements from $A_{\pm}$. We write $H_{P,k}^* = \{ n \in \mathbb{Z} : n = \sum_{j=1}^l a_j ~ {\rm  with} ~  a_j \in A_{\pm}, l \leq k \}.$ Choose the smallest positive integer $i>10$ such that $j > (\log j)^{C_5kt \log \log \log j}$ for $j\geq i$. Then $i < 2(\log i)^{C_5kt \log \log \log i}$. It follows that
$$
\log i  < C_8 kt (\log \log i)( \log \log \log i)
$$
for some constant $C_8$.
Hence $\log i < C_9kt ( \log (kt))( \log \log (kt))$ for some
constant $C_9$. According to Theorem \ref{thm2} there exists an integer $m$ with $ \log i \leq \log m \leq (\log i )^{C_3}$ such that all representations in $H_{P,k}^*$  are covered by at most  $ (\log m)^{C_5kt \log \log \log m}$ residue classes modulo $m$. By the definition of $i$ and the inequality $ i \leq m$, we see that this number of residue classes is less than $m$, therefore at least one positive integer $n \leq m$ has no representation of the form $\sum_{j=1}^{k} a_j$ with $a_j\in A\cup \{0\}$ for $j=1,\dots,k$. Hence
$$
\log n \leq \log m \leq ( \log i )^{C_3} < \left( C_9 kt ( \log kt)( \log \log kt)\right)^{C_3}  < (kt)^{C_{10}}
$$
for some constant $C_{10}$. Thus $F_{\pm} (k) < \exp((kt)^{C_{10}}).$

So we have proved the following result.

\vskip.3cm

\begin{thm}
\label{thm2} Let $P=\{p_1, \dots, p_t\}$ be a finite set of primes with $t \geq 2$. Let $A$ be the set of integers composed of numbers from $P$. Let $k$ be a positive integer. Denote by $F(k)$ the smallest positive integer which cannot be represented in the form $ \sum_{i=1}^{k} a_i$ with $a_i \in A\cup \{0\}$ for all $i$ and by $F_{\pm}(k)$ the smallest positive integer which cannot be represented in the form $ \sum_{i=1}^{k}
a_i$ with $a_i \in A_{\pm}\cup \{0\}$ for all $i$. Then, for every $\varepsilon > 0$ there are a number $c$ depending only on the two smallest elements of $P$, a number $C$ depending only on $\varepsilon$ and an absolute constant $C_{\pm}$ such that\\
(i)  $ F(k) > k^{ck}$ for all $k>1$, \\
(ii)  $F(k) \leq C(kt)^{(1 + \varepsilon)kt}$ for all $k>1$, \\
(iii) $ F_{\pm} (k) < \exp((kt)^{C_{\pm}})$ for all $k>1$. \\
\end{thm}

\noindent {\bf Remark 1.}
In Section 5 we shall use an ineffective method to show that $C_{\pm} =16$ suffices.
\\
\\
{\bf Remark 2.} Following the proof of Theorem 3(iv) of \cite{ht} it can be shown that there are infinitely many positive integers $k$ for which $F_{\pm}(k) \leq \exp (C^*_{\pm} kt  \log(kt) \log \log (kt))$ for some suitable effectively computable constant $C^*$. In Section 5 we derive the better upper bound $(kt)^{(1 + \varepsilon)kt}$ for $F_{\pm}(k)$ for all but finitely many $k$. However, it cannot be deduced from the proof from which value of $k$ on this bound holds.
\\
\\
{\bf Remark 3.} Using the above methods similar bounds can be derived if $P$ is replaced by any finite set of positive integers.

\section{Application of the ineffective Subspace theorem}

By applying another version of the Subspace Theorem we derive an estmate for $F_{\pm}(k)$ which is much better than the bound in Theorem 2(iii) and holds for all but finitely many $k$'s.
\begin{thm}
\label{thm4} Under the conditions of Theorem $2$ for every $\varepsilon > 0$ there is a number $c^*_{\pm}$ depending only on $P,k$ and $\varepsilon$ such that
$$ F_{\pm}(k) \leq (kt)^{(1 + \varepsilon)kt}$$ whenever $k >c^*_{\pm}$.
\end{thm}

\noindent In the proof we apply the following result of Evertse. Here the $p$-adic value $|x|_p$ is defind as $|x| p^{-r}$ where $p^r || x $.

\begin{lem} [\cite{ev}, Corollary 1] \label{ev}
Let $c,d$ be constants with $c>0, 0 \leq d < 1$. Let $S_0$ be a finite set of primes and let $l$ be a positive integer.
Then there are only finitely many tuples $ (x_0, x_1, \dots, x_l)$ of rational integers such that
$$ x_0 + x_1 + \dots + x_l = 0; $$
$$x_{i_0} + x_{i_1} + \dots x_{i_s} \not= 0 $$
for each proper, non-empty subset $\{i_0, i_1, \dots, i_s\}$ of $ \{0, 1,  \dots, l \};$
$$\gcd (x_0, x_1, \dots, x_l) =1 ;$$
$$\prod_{j=0}^{l} \left( |x_j| \prod_{p \in S_0} |x_j|_p \right) \leq c  \left( \max_{0 \leq j \leq l} |x_j| \right)^d.$$
\end{lem}

\begin{proof}[Proof of Theorem $4$.] Let $n$ be an integer which is not divisible by any prime from $P$.
Suppose $n=a_1 + a_2 + \dots + a_l$ with $a_j \in A_{\pm} $ for $j=1, 2, \dots, l$ with $l \leq k$.
Without loss of generality we may assume that $l$ is minimal, hence $a_1 + a_2 + \dots + a_l$ has no proper subsums which vanish.
Moreover, we know that $\gcd (a_1, a_2, \dots, a_l) =1$.
We apply Lemma \ref{ev} with $c=1, d=1/2, S_0=P$ to the equation $a_0 + a_1 + \dots + a_l=0$ with $a_0 = -n$.
It follows that given $k, P$ there only finitely many tuples $(n, a_1, a_2, \dots, a_l)$ with $\gcd(n,p_1, \dots, p_t) =1$ and $l \leq k$
such that $n=a_1 + a_2 + \dots + a_l$ with $a_j \in A_{\pm}$ for $j=1, 2, \dots, l$ and
$$n  \leq \left( \max_{0 \leq j \leq l} | a_j | \right)^{1/2},$$
hence
$$ n^2 \leq \max_{1 \leq j \leq l} | a_j |.$$
Let $N_0$ be the maximum of $|n|$ for all such tuples, where $N_0=0$ if there are no such tuples.

Next consider positive integers $n>N_0$ which are not divisible by any prime from $P$.
Then, for any representation $n=a_1 + a_2 + \dots + a_l$ with $a_j \in A_{\pm}$ for $j=1, 2, \dots, l$ and $l \leq k$,
we have $ | a_j | < n^2$ for $j=1,2, \dots, l$. Writing $a_j = \pm p_1^{s_1} \cdots p_t^{s_t}$ we obtain $\max_j s_j \leq 3 \log n -1. $
The number of possible tuples $(a_1, \dots, a_l)$  for $l$ is therefore at most $2^l (3 \log n)^{lt} $. Then the number of all possible tuples $(a_1,...,a_j)$ with $j\leq k$ is at most
$2\cdot 2^k(3\log n)^{kt}$. Thus for $N>N_0$ there are at most $N_0 + 2\cdot 2^k (3 \log N)^{kt}$ integers $n\leq N$ coprime to $P$ such that $n$ is representable as sum of at most $k$ integers from $A_{\pm}$. The number of positive integers $n \leq N$ coprime to $P$ is at least $N\prod_{p \in P} (1-1/p) - 2^t> 2^{-t}N-2^t$. Hence for finding an $n$ with $n \leq N$ such that $n$ is not representable in the desired form, it suffices that
$$
2^{-t}N-2^t > N_0 + 2\cdot 2^k (3 \log N)^{kt}.
$$
As in the proof of Theorem 2(ii) it follows that for every $\varepsilon > 0$ there is an unspecified number $c^*_{\pm}$ depending only on $k, P$ and $\varepsilon$ such that
$$ F_{\pm} (k) \leq (kt)^{(1 + \varepsilon)kt} $$ whenever $k>c^*_{\pm}$.
\end{proof}

\noindent {\bf Remark 4.} Theorem 4 is also an improvement of Theorem 3.4(iv) of \cite{ht} where, only for sums of perfect powers, a weaker bound is given.


\begin{thebibliography}{8888}

\bibitem{ahl} Zs. \'Ad\'am, L. Hajdu, F. Luca, Representing
integers as linear combinations of $S$-units, Acta Arith. {\bf
138} (2009), 101--107.

\bibitem{eps} P. Erd\H{o}s, C. Pomerance, E. Schmutz,
Carmichael's lambda function, Acta Arith. {\bf 58} (1991),
365--385.

\bibitem{ev} J.-H. Evertse, On sums of $S$-units and linear
recurrences, Compositio Math. {\bf 53} (1984), 225--244.

\bibitem{h} L. Hajdu, Arithmetic progressions in linear
combinations of $S$-units, Period. Math. Hungar. {\bf 54}
(2007), 175--181.

\bibitem{ht} L. Hajdu, R. Tijdeman, Representing integers as
linear combinations of powers, Publ. Math. Debrecen {\bf 79} (2011), 461--468.

\bibitem{jn} M. Jarden, W. Narkiewicz, On sums of units,
Monatsh. Math. {\bf 150} (2007), 327--332.

\bibitem{n} M. B. Nathanson, Geometric group theory and
arithmetic diameter, Publ. Math. Deb\-recen {\bf 79} (2011),
563--572.

\bibitem{t} R. Tijdeman, On the maximal distance of integers
composed of small primes, Compos. Math. {\bf 28} (1974), 159--162.

\end{thebibliography}
\end{document}